\documentclass[oneside]{amsart}
\usepackage{pstricks}
\usepackage{amsmath}
\usepackage{amsfonts}
\usepackage{amsthm}
\usepackage{hyperref}
\usepackage{amscd}
\usepackage{color}
\usepackage{enumerate}
\usepackage{subcaption}
\usepackage{tikz}
\usetikzlibrary{intersections}
\DeclareCaptionSubType[Alph]{figure}

\newcommand{\KK}{\mathbb{K}}
\newcommand{\PP}{\mathbb{P}}

\newcommand{\A}{\mathcal{A}}
\newcommand{\Aff}{\mathbb{A}}
\newcommand{\RR}{\mathbb{R}}

\newcommand{\NN}{\mathbb{N}}
\newcommand{\CO}{\mathcal{O}}
\newcommand{\I}{\mathcal{I}}
\newcommand{\K}{\mathcal{K}}

\newcommand{\mcL}{\mathcal{L}}
\newcommand{\mcP}{\mathcal{P}}
\newcommand{\Q}{\mathcal{Q}}
\newcommand{\R}{\mathcal{R}}
\newcommand{\GR}{\G(2,n)}
\newcommand{\G}{G}

\DeclareMathOperator{\Def}{Def}
\DeclareMathOperator{\st}{st}
\DeclareMathOperator{\Hom}{Hom}
\DeclareMathOperator{\spec}{Spec}
\DeclareMathOperator{\Proj}{Proj}
\DeclareMathOperator{\Circ}{circ}

\DeclareMathOperator{\Der}{Der}
\DeclareMathOperator{\depth}{depth}

\newtheorem{thm}{Theorem}[section]
\newtheorem{prop}[thm]{Proposition}
\newtheorem{lemma}[thm]{Lemma}
\newtheorem{cor}[thm]{Corollary}

\theoremstyle{definition}

\newtheorem{rem}[thm]{Remark}
\newtheorem{case}{Case}

\newtheorem{ex}[thm]{Example}

\title[Degenerations for Dual Quotient Bundles]{Unobstructed Stanley-Reisner Degenerations for Dual Quotient Bundles on $\GR$}

\author{Nathan Ilten}
\address{Department of Mathematics, Simon Fraser University,
8888 University Drive, Burnaby BC V5A1S6, Canada}
\email{nilten@sfu.ca}
\thanks{The first author was partially supported by an NSERC Discovery Grant.} 

\author{Charles Turo}
\address{Department of Mathematics, Simon Fraser University,
8888 University Drive, Burnaby BC V5A1S6, Canada}
\email{cturo@sfu.ca}
\thanks{The second author was partially supported by an NSERC USRA Fellowship.}

\newcommand{\drawksix}{
\begin{tikzpicture}[every node/.style={draw,shape=circle}]
\path (0,0) node (t1) {$x_{13}$}
(4,0) node (t2) {$x_{15}$}
(8,0) node (t3) {$x_{35}$}
(0,-2) node (b1) {$x_{46}$}
(4,-2) node (b2) {$x_{24}$}
(8,-2) node (b3) {$x_{26}$}
(12,0) node[fill=lightgray] (t4) {$x_{13}$}
(12,-2) node[fill=lightgray] (b4) {$x_{46}$}
(2,-1) node (m1) {$x_{14}$}
(6,-1) node (m2) {$x_{25}$}
(10,-1) node (m3) {$x_{36}$}
(6,4) node (y1) {$y_1$}
(6,-5) node (yn) {$y_6$}
(5,2) node (x) {$x_{16}$};
\draw[name path=bottom] (yn) -- (t1) 
(yn) -- (t2)
(yn) -- (t3)
(yn) -- (t4)
(yn) -- (m1)
(yn) -- (m2)
(yn) -- (m3)
(yn) -- (b1)
(yn) -- (b2)
(yn) -- (b3)
(yn) -- (b4);
\draw[name path=middle] (t1) -- (t2)
(t2) -- (t3)
(t3) -- (t4)
(b1) -- (b2)
(b2) -- (b3)
(b3) -- (b4)
(t1) -- (b1)
(t2) -- (b2)
(t3) -- (b3)
(t4) -- (b4)
(m1) -- (t1)
(m1) -- (b1)
(m1) -- (t2)
(m1) -- (b2)
(m2) -- (t2)
(m2) -- (b2)
(m2) -- (t3)
(m2) -- (b3)
(m3) -- (t3)
(m3) -- (b3)
(m3) -- (t4)
(m3) -- (b4);
\fill[white,name intersections={of=middle and bottom,total=\t}]
    \foreach \s in {1,...,\t}{(intersection-\s) circle (4pt) {}};
\draw[name path=middle] (t1) -- (t2)
(t2) -- (t3)
(t3) -- (t4)
(b1) -- (b2)
(b2) -- (b3)
(b3) -- (b4)
(t1) -- (b1)
(t2) -- (b2)
(t3) -- (b3)
(t4) -- (b4)
(m1) -- (t1)
(m1) -- (b1)
(m1) -- (t2)
(m1) -- (b2)
(m2) -- (t2)
(m2) -- (b2)
(m2) -- (t3)
(m2) -- (b3)
(m3) -- (t3)
(m3) -- (b3)
(m3) -- (t4)
(m3) -- (b4);
\draw[name path=topa] 
(x) -- (t1) 
(x) -- (t2)
(x) -- (t3)
(x) -- (m1)
(x) -- (m2)
(x) -- (b2);
\fill[white,name intersections={of=topa and bottom,total=\t}]
    \foreach \s in {1,...,\t}{(intersection-\s) circle (4pt) {}};
\fill[white,name intersections={of=topa and middle,total=\t}]
    \foreach \s in {1,...,\t}{(intersection-\s) circle (4pt) {}};
\draw[name path=topa] 
(x) -- (t1) 
(x) -- (t2)
(x) -- (t3)
(x) -- (m1)
(x) -- (m2)
(x) -- (b2);
\draw[name path=topb] (y1) -- (t1) 
(y1) -- (x)
(y1) -- (t3)
(y1) -- (t4)
(y1) -- (m1)
(y1) -- (m2)
(y1) -- (m3)
(y1) -- (b1)
(y1) -- (b2)
(y1) -- (b3)
(y1) -- (b4);
\fill[white,name intersections={of=topb and bottom,total=\t}]
    \foreach \s in {1,...,\t}{(intersection-\s) circle (4pt) {}};
\fill[white,name intersections={of=topb and middle,total=\t}]
    \foreach \s in {1,...,\t}{(intersection-\s) circle (4pt) {}};
\fill[white,name intersections={of=topb and topa,total=\t}]
    \foreach \s in {1,...,\t}{(intersection-\s) circle (4pt) {}};
\fill[white] (5.4,2.2) circle (3pt) {};
\fill[white] (5.2,1.7) circle (3pt) {};

\draw[name path=topb] (y1) -- (t1) 
(y1) -- (x)
(y1) -- (t3)
(y1) -- (t4)
(y1) -- (m1)
(y1) -- (m2)
(y1) -- (m3)
(y1) -- (b1)
(y1) -- (b2)
(y1) -- (b3)
(y1) -- (b4);
\draw (x)--(b2);
\end{tikzpicture}
}
\newcommand{\drawasix}{
\begin{tikzpicture}[every node/.style={draw,shape=circle}]
\path (0,0) node (t1) {$x_{13}$}
(4,0) node (t2) {$x_{15}$}
(8,0) node (t3) {$x_{35}$}
(0,-2) node (b1) {$x_{46}$}
(4,-2) node (b2) {$x_{24}$}
(8,-2) node (b3) {$x_{26}$}
(12,0) node[fill=lightgray] (t4) {$x_{13}$}
(12,-2) node[fill=lightgray] (b4) {$x_{46}$}
(2,-1) node (m1) {$x_{14}$}
(6,-1) node (m2) {$x_{25}$}
(10,-1) node (m3) {$x_{36}$};
\draw (t1) -- (t2)
(t2) -- (t3)
(t3) -- (t4)
(b1) -- (b2)
(b2) -- (b3)
(b3) -- (b4)
(t1) -- (b1)
(t2) -- (b2)
(t3) -- (b3)
(t4) -- (b4)
(m1) -- (t1)
(m1) -- (b1)
(m1) -- (t2)
(m1) -- (b2)
(m2) -- (t2)
(m2) -- (b2)
(m2) -- (t3)
(m2) -- (b3)
(m3) -- (t3)
(m3) -- (b3)
(m3) -- (t4)
(m3) -- (b4);
\end{tikzpicture}
}

\begin{document}

\begin{abstract}
Let $\Q^*$ denote the dual of the quotient bundle on the Grassmannian $G(2,n)$. We prove that the ideal of $\Q^*$ in its natural embedding has initial ideal equal to the Stanley-Reisner ideal of a certain unobstructed simplicial complex. Furthermore, we show that the coordinate ring of $\Q^*$ has no infinitesimal deformations for $n>5$. 
\end{abstract}
\maketitle

\section{Introduction}
Let $\GR$ denote the Grassmannian parametrizing $2$-dimensional linear subspaces of an $n$-dimensional vector space, and let $I_{2,n}$ be the ideal of this variety in its Pl\"ucker embedding. Denote by $\A_n$ the boundary complex of the dual polytope of the $n$-associahedron. In \cite[Proposition 3.7.4]{sturmfels:08a}, Sturmfels shows that the Stanley-Reisner ideal associated to the join of  $\A_n$ with an $(n-1)$-dimensional simplex is an initial ideal of $I_{2,n}$.
As the boundary complex of a polytope, $\A_n$ has the nice property that it is pure-dimensional and  topologically a sphere. Furthermore, Christophersen and the first author have shown \cite[Theorem 5.6]{ilten:14a} that $\A_n$ is \emph{unobstructed}, that is, the second cotangent cohomology of the associated Stanley-Reisner ring vanishes. This fact can be used to construct new degenerations of $\GR$ to certain toric varieties.

In \cite{ilten:14a} it is also shown that a similar situation holds for a number of other classical varieties, including the orthogonal Grassmannian $SO(5,10)$ of isotropic $5$-planes in a $10$-dimensional vector space. However, no such construction is apparent for other orthogonal Grassmannians. In this paper we generalize in a different direction. The key observation is that the coordinate ring of $SO(5,10)$ in its Pl\"ucker embedding is in fact a deformation of the coordinate ring of the dual of the tautological quotient bundle on $\G(2,5)$, embedded in $\PP^9\times \Aff^5$.

Let $\Q$ denote the tautological quotient bundle on $\GR$, and $\Q^*$ its dual, see \S \ref{sec:Q}. The bundle $\Q^*$ comes with a natural embedding in $\GR\times \Aff^n$, and hence (after composition with the Pl\"ucker embedding) an embedding in $\PP^{ { {n}\choose{2}}-1}\times\Aff^n$. Let $J_n$ denote the ideal of $\Q^*$ in this embedding. Our first main result (Theorem \ref{thm:1}) is to show that under an appropriate term order, the initial ideal of $J_n$ is the Stanley-Reisner ideal associated to the join of a $2n-4$-dimensional simplex with a simplicial complex $\K_n$. This complex $\K_n$ is obtained from the bipyramid over $\A_n$ via a stellar subdivision, see \S \ref{sec:assoc} for details. It shares many of the nice properties of $\A_n$: topologically it is a sphere, and by a result of \cite{ilten:12a} it is also unobstructed.

Motivated by the fact that for $n=5$, the coordinate ring of $\Q^*$ deforms to that of $SO(5,10)$, we consider the deformation theory of the coordinate ring $S_n/J_n$ of $\Q^*$. Our second main result (Theorem \ref{thm:2}) is that $T^1(S_n/J_n)=0$ if $n>5$, that is, $S_n/J_n$ has no infinitesimal deformations. The proof of this fact relies heavily on our first result. Along the way, we describe the syzygies of $J_n$ in \S \ref{sec:syz}, and 
obtain an algebraic proof of the result of Svanes \cite{svanes:75a} that the coordinate ring of $\GR$ has no infinitesimal deformations for $n\geq 5$. 

Our motivation for the above two results comes from our desire to better understand higher-dimensional Fano varieties, along with their degenerations to toric varieties.
The bundle $\Q^*$ is related in a natural fashion to two projective varieties: its projectivization $\PP(\Q^*)$, and the closure of $\Q^*$ in an appropriate projective space, see \S \ref{sec:geom}. Complete intersections of sufficiently low degree in both these varieties are Fano, and they come equipped with degenerations to unobstructed Stanley-Reisner schemes due to our result above on the initial ideal of $J_n$. This can be used as in \cite{ilten:14a} to find toric degenerations of these Fano varieties. Moreover, we are able to use our rigidity result on $S_n/J_n$ to show that both the projectivization and the projective closure of $\Q^*$ are rigid, and that any class of hypersurfaces in one of these varieties is closed under deformation, see Theorem \ref{thm:g}.

\subsection*{Acknowledgements} We thank Jan Christophersen for answering several questions.
\section{Preliminaries}
\subsection{Simplicial complexes and Stanley-Reisner ideals}
Let $V$ be any finite set, and $\mcP(V)$ its power set. An abstract \emph{simplicial complex} on $V$ is any subset $\K\subset\mcP(V)$ such that if $f\in\K$ and $g\subset f$, then $g\in \K$. Elements $f\in\K$ are called \emph{faces}; the dimension of a face $f$ is $\dim f:=\#f-1$. Zero-dimensional faces are called \emph{vertices}; one-dimensional faces are called \emph{edges}. By $\Delta_n$ we denote $n$-dimensional simplex, the power set of $\{0,\ldots,n\}$.
By $\Delta_{-1}$ we denote the empty set. By $\partial \Delta_n$ we denote the boundary of $\Delta_n$, that is, $\partial \Delta_n=\Delta_n\setminus \{0,\ldots,n\}$.

Fix an algebraically closed field $\KK$. Let $\K\subset\mcP(V)$ be any simplicial complex. Its \emph{Stanley-Reisner ideal} is the square-free monomial ideal $I_\K\subset \KK\left[x_i\ |\ i\in V\right]$
$$
I_\K:=\langle x_p \ | \ p\in\mcP(V)\setminus\K\rangle
$$
where for $p\in\mcP(V)$, $x_p:=\prod_{i\in p}x_i$. 
This gives rise to the \emph{Stanley-Reisner ring} $A_\K:=\KK\left[x_i\ |\ i\in V\right]/I_\K$.
We refer to \cite{stanley:83a} for more details on Stanley-Reisner theory.

We will be interested in two operations on simplicial complexes.
Given simplicial complexes $\K$ and $\mcL$ on disjoint sets, their \emph{join} is the simplicial complex
$$
\K * \mcL=\{f\cup g\ | \ f\in\K,\ g\in\mcL\}.
$$
Given a simplicial complex $\K$ and a face $f\in\K$ of dimension at least one,
the \emph{stellar subdivision} of $\K$ at $f$ is the simplicial complex
$$
\st(f,\K)=\{g\in\K\ |\ f\not\subseteq g\}\cup\{g\cup v\in\K*\Delta_0\ |\ f\not\subseteq g\ \textrm{and}\ f\cup g\in\K\}
$$
contained in $\K*\Delta_0$, where $v$ is the single vertex of $\Delta_0$. Roughly speaking, stellar subdivision at $f$ splits the $k$-simplex $f$ into $k+1$ new $k$-simplices by inserting the vertex $v$. Any simplex of $\K$ containing $f$ is replaced by $k+1$ copies containing the new pieces of $f$. We will especially be interested in the case when $f$ is an edge. If $\K'$ is a simplicial complex obtained from $\K$ by stellar subdivision at some edge, we say that $\K$ is an \emph{edge-unstarring} of $\K'$.

\subsection{Associahedra and related complexes}\label{sec:assoc}
We will now consider some simplicial complexes related to triangulations of the $n$-gon.
Consider the $n$-gon and index the vertices in cyclical order by $i =
1, \dots , n$. Denote by $\delta_{ij}$ the diagonal between vertex $i$
and vertex $j$, and let $V_n$ be the set of all $\frac{1}{2} n (n-3)$ diagonals.
We now define the simplicial complex $\A_n$ to be the complex on $V_n$ whose faces are sets of pairwise non-crossing diagonals. Hence,  $r$-dimensional faces of $\A_n$ corresponds to a partition of the $n$-gon into a union of $r+2$ polygons. In particular, maximal faces correspond to triangulations of the $n$-gon.
This simplicial complex is the boundary complex of the dual polytope of the associahedron, which plays a role in
many fields. For more details, see e.g. \cite{lee:89a} and the introduction in \cite{hohlweg:07a}.
For small $n$ we have $\A_3 = \{\emptyset\}$, $\A_4 $ is two vertices
with no edge, and $\A_5$ is the boundary of the
pentagon.

It is straightforward to describe the Stanley-Reisner ideal $I_{\A_n}$. The ideal 
$I_{\A_n}$ is generated by the $n\choose 4$ quadrics of the form $x_{ik}x_{jl}$ for $1\leq i < j < k < l\leq n$, corresponding to crossing diagonals in the $n$-gon. Here, for ease of notation, we have written $x_{ik}$ for the variable $x_{\delta_{ik}}$.

\begin{ex}\label{ex:1}
We picture the edges of the simplicial complex $\A_6$ in Figure \ref{fig:a6}. In order to more easily draw this complex, we have positioned the vertices $x_{13}$ and $x_{46}$ both on the left and right hand sides of the illustration. Note that the complex $\A_n$ is a \emph{flag complex} (minimal non-faces are one-dimensional) and is hence determined by its edges.
\end{ex}

\begin{figure}
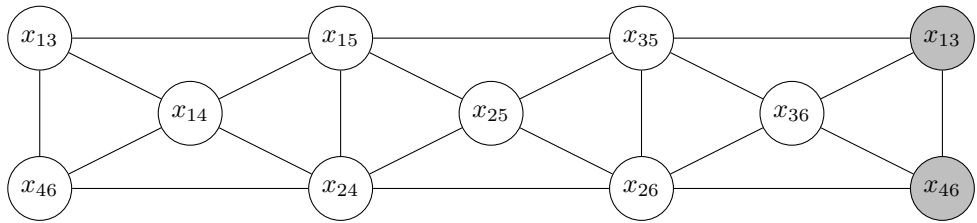

\drawasix
\caption{The simplicial complex $\A_6$}\label{fig:a6}
\end{figure}

We now define a related complex which plays a central role in this paper. Fix a natural number $n\geq 5$, and let $S^0$ be the simplicial complex with vertices $w_1,w_2$ and no edge. We define the complex $\K_n$ as
\[
\K_n=\st(f,\A_n*S^0)
\]
where $f$ is the edge of $\A_n*S^0$ with vertices $\delta_{1(n-1)}$ and $w_1$.

\begin{ex}
In Figure \ref{fig:k6}, we picture the edges of the simplicial complex $\K_6$. As in Example \ref{ex:1}, we draw the vertices $x_{13}$ and $x_{46}$ on both the left and right to facilitate illustration. The complex $\K_6$ is again determined by the edges pictured here, since $\K_n$ is also a flag complex.
\end{ex}
\begin{figure}
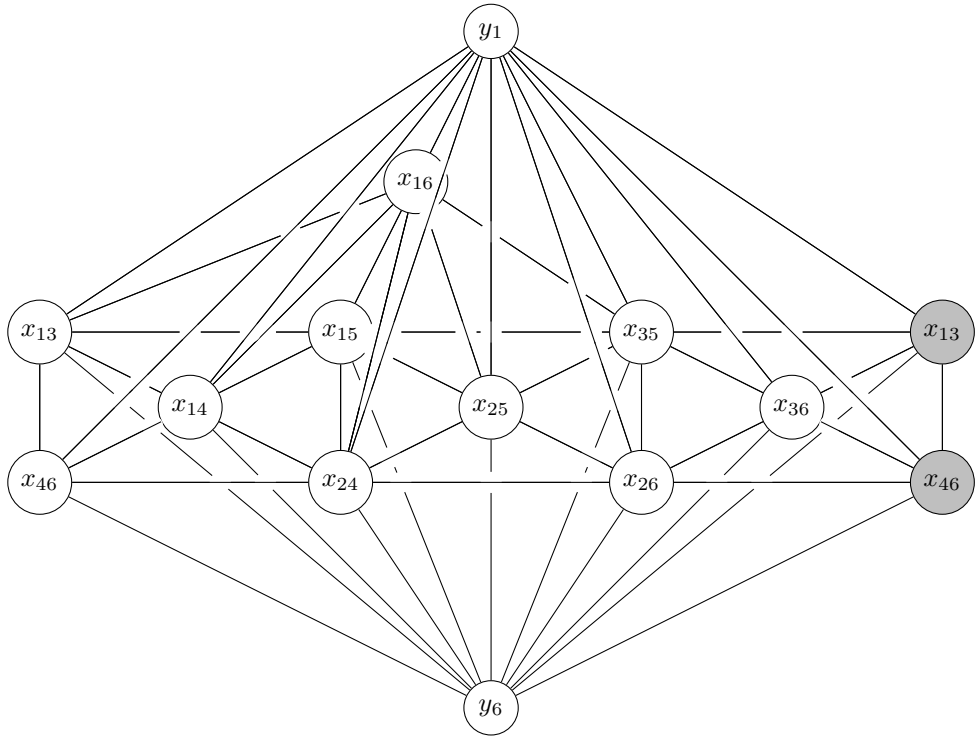

\drawksix
\caption{The simplicial complex $\K_6$}\label{fig:k6}
\end{figure}

It is also straightforward to describe the Stanley-Reisner ideal $I_{\K_n}$.
 Set $y_{1}=x_{w_1}$, $y_n=x_{v}$,  $x_{1n}=x_{w_2}$, where $v$ is the new vertex from the stellar subdivision.
\begin{lemma}
The ideal 
$I_{\K_n}$ is generated by the $n\choose 4$ quadrics of the form $x_{ik}x_{jl}$ for $1\leq i < j < k < l\leq n$, together with the quadrics 
\[
x_{1(n-1)}y_1,\ x_{1n}y_1,\ x_{1n}y_n,\ x_{2n}y_n,\ \ldots,\ x_{(n-2)n}.
\]
\end{lemma}
\begin{proof}
The minimal non-faces of $\A_n*S^0$ are exactly the minimal non-faces of $\A_n$, together with the non-face $\{w_1,w_n\}$. The stellar subdivision in $f$ preserves all of these non-faces, as well as creating the new minimal non-faces $\{w_2,v\}$, $\{\delta_{1,n-1},w_1\}$, and $\{\delta_{i,n},v\}$ for $i=2,\ldots n-2$.
\end{proof}

In additional to both being topological spheres, $\A_n$ and $\K_n$ share the additional important property of being unobstructed. For any $\K$-algebra $A$, let $T^i(A)$ denote the \emph{cotangent cohomology} module $T^i(A/\KK,A)$, see e.g.~\cite{andre:74a} for details on cotangent cohomology.
\begin{thm}[\cite{ilten:12a,ilten:14a}]\label{thm:unobstructed}
For $n\geq 5$, the simplicial complexes $\A_n$ and $\K_n$ are unobstructed, that is, 
\[T^2({A_{\A_n})}=T^2(A_{\K_n})=0.\]
\end{thm}
\begin{proof}
For $\A_n$, this is \cite[Theorem 5.6]{ilten:14a}. For $\K_n$, this follows from \cite[Theorem 6.2]{ilten:12a}. Indeed, the complex $\K_n$ arises as the  final intermediate simplicial complex in the sequence of $n-3$ edge-unstarrings going from $\A_{n+1}$ to $\A_{n}*\A_{4}$ as described in loc.~cit.
\end{proof}

\subsection{Equations for $\GR$ and $\Q^*$}\label{sec:Q}
For $n\geq 4$, let $I_{2,n}$ denote the ideal of the Grassmannian $\GR$ in its Pl\"ucker embedding. 
Recall that if $x_{ij}$ $1\leq i<j\leq n$ are the standard Pl\"ucker coordinates, 
$I_{2,n}$ is generated by the $4\times 4$ Pfaffians of the $n\times n$ skew-symmetric matrix with coordinates
\begin{equation}\label{eqn:matrix}
\left(\begin{array}{c c c c c}
0&x_{12}&x_{13}&\cdots&x_{1n}\\
-x_{12}&0&x_{23}&\cdots&x_{2n}\\
-x_{13}&-x_{23}&0&\cdots&x_{3n}\\
\vdots&\vdots&\vdots&\ddots&\vdots\\
-x_{1n}&-x_{2n}&-x_{3n}&\cdots&0
\end{array}
\right),
\end{equation}
see e.g. \cite[\S 3]{sturmfels:08a}.
For $i< j < k < l$ we denote the Pfaffian involving the rows $i,j,k,l$ by $\Phi_{ijkl}$. Explicitly, we have
\[
\Phi_{ijkl}=x_{ij}x_{kl}-x_{ik}x_{jl}+x_{il}x_{jk}.
\]

Consider any monomial order $\prec_{\Circ}$ which, for $1\leq i<j<k<l\leq n$, selects the monomial $x_{ik}x_{jl}$ as the lead term in the Pfaffian $\Phi_{ijkl}$. Such terms orders exist and are called \emph{circular}; the Pfaffians form a Gr\"obner basis for these term orders.  The initial ideal of $I_{2,n}$ is exactly the Stanley-Reisner ideal of $\A_n*\Delta_{n-1}$, where the vertices of $\Delta_{n-1}$ correspond to the variables $x_{1n},x_{12},x_{23},\ldots,x_{(n-1)n}$.
 See \cite[Proposition 3.7.4]{sturmfels:08a} for details.

On $\G(2,n)$, there is a tautological sequence of vector bundles:
\[
\begin{CD}
0 @>>> \R @>>> \CO_{\GR}^n @>>> \Q @>>> 0.
\end{CD}
\]
$\R$ is the \emph{tautological subbundle} of rank $2$ whose fiber at a point $p\in \GR$ is exactly the $2$-plane parametrized by the point $p$. Each such $2$-plane sits inside $\KK^n$, leading to the inclusion $\R\hookrightarrow \CO_{\GR}^n$.
The \emph{tautological quotient bundle} $\Q$ is the rank $n-2$ bundle given fiberwise by the quotient of $\KK^n$ by the $2$-plane parametrized by a point $p$. Taking the dual of the above exact sequence gives an inclusion of the \emph{dual quotient bundle} $\Q^*$ in ${(\CO_{\GR}^n)}^*$. Hence, we can geometrically realize the bundle $\Q^*$ in $G(2,n)\times \Aff^n$. After composing with the Pl\"ucker embedding, we obtain $\Q^*$ as a subvariety of $\PP^{ {n\choose 2}-1}\times\Aff^n$. We wish to describe the ideal of $\Q^*$ in this embedding. Later, our first main result (Theorem \ref{thm:1}) will be to show that under a suitable term order, its initial ideal equals the Stanley-Reisner ideal of $\K_n*\Delta_{2n-4}$.

As before we let $x_{ij}$ denote the standard Pl\"ucker coordinates on $\PP^{ {n\choose 2}-1}$; let $y_1,\ldots,y_n$ be coordinates for $\Aff^n$, and consider the ring
\[S_n= \KK[x_{ij}, y_1,\ldots,y_n]_{1\leq i < j \leq n}.\]
 Define quadrics $f_1,\ldots,f_n$ by 
\begin{equation*}
\left(\begin{array}{c c c c c}
0&x_{12}&x_{13}&\cdots&x_{1n}\\
-x_{12}&0&x_{23}&\cdots&x_{2n}\\
-x_{13}&-x_{23}&0&\cdots&x_{3n}\\
\vdots&\vdots&\vdots&\ddots&\vdots\\
-x_{1n}&-x_{2n}&-x_{3n}&\cdots&0
\end{array}
\right)\cdot
 \left(\begin{array}{c}
y_1\\
y_2\\
y_3\\
\vdots\\
y_n
\end{array}\right)
=
 \left(\begin{array}{c}
f_1\\
f_2\\
f_3\\
\vdots\\
f_n
\end{array}\right).
\end{equation*}
Let $J_n\subset S_n$ be the ideal generated by the $n\choose 4$ Pfaffians $\Phi_{ijkl}$ along with the $n$ additional quadrics $f_1,\ldots,f_n$.

\begin{prop}
The ideal of $\Q^*$ is $J_n$. In particular, $J_n$ is prime.
\end{prop}
\begin{proof}
We first show that $J_n$ cuts out $\Q^*$ set-theoretically. Let $p$ be a point in $\GR$. The $2$-plane $P$ in $\KK^n$ corresponding to $p$ is simply the rowspan of the matrix \eqref{eqn:matrix}, after substituting in the Pl\"ucker coordinates for $p$. On the other hand, the fiber of $\Q^*$ over $p$ consists of those linear functionals $\KK^n\to \KK$ which vanish on $P$. But if $e_1,\ldots,e_n$ is the standard basis of $\KK^n$, requiring that a functional $\sum y_i e_i^*$ vanish on $P$ is the same as requiring that $f_1,\ldots,f_n$ vanish at $p$. Hence, the ideal $J_n$ cuts out $\Q^*$ set-theoretically. 

That $J_n$ is the (scheme-theoretic) ideal of $\Q^*$ follows from the fact that $J_n$ is a prime ideal. To prove this, consider the affine scheme $V(J_n)$ in $\Aff^{n\choose 2}\times\Aff^n$. A straightforward calculation with the Jacobian of $J_n$ shows that the singular locus of $V(J_n)$ is a proper subset of $V(J_n)$, hence $V(J_n)$ is generically reduced. Furthermore, we shall see in Theorem \ref{thm:1} that an initial ideal of $J_n$ is the Stanley-Reisner ideal of $\K_n*\Delta_{2n-4}$. Since $\K_n$ is a topological sphere and hence shellable, it follows that $S_n/J_n$ is Cohen-Macaulay, see e.g. \cite[Corollary 8.31 and Theorem 13.45]{miller:05a}.
But since the scheme $V(J_n)$ is generically reduced and irreducible, it must be integral, see e.g. \cite[Theorem 18.15]{eisenbud:95a}.
Hence, $J_n$ is prime.
\end{proof}

\begin{rem}
It also seems natural to consider the ideal of the tautological subbundle $\R$ in  $\PP^{ {n\choose 2}-1}\times\Aff^n$, and connect it to some simplicial complex with desirable properties. However, an easy calculation shows that the ideal of $\R$ is the ideal $I_{2,n+1}$, and hence, its initial ideal under a circular term order is just the Stanley-Reisner ideal of $\A_{n+1}*\Delta_n$. We leave the details to the reader.
\end{rem} 

\subsection{Stanley-Reisner Schemes}
This section is only necessary for those readers interested in the geometric applications of our results, found in \S \ref{sec:geom}.
Given a simplicial complex $\K$, we let $\PP(\K)$ be the vanishing locus of $I_\K$ in the appropriate projective space. In other words, we have
\[\PP(\K):=\Proj A_\K.\] We call $\PP(\K)$ the \emph{Stanley-Reisner scheme} of $\K$.
Note that the
 scheme $\PP(\K)$ ``looks'' like the complex $\K$: each face $f\in\K$ corresponds to some $\PP^{\dim f}\subset \PP(\K)$ and the intersection relations among these projective spaces are identical to those of the faces of $\K$. In particular, maximal faces of $\K$ correspond to the irreducible components of $\PP(\K)$.

We will also be interested in subschemes of products of projective spaces. Suppose that $\K$ is a simplicial complex on $V$, and we have partitioned $V$ into two disjoint sets $V_1$ and $V_2$. We then consider variables $x_i$, $i\in V_1$ as coordinates on $\PP^{|V_1|-1}$ and $x_i$, $i\in V_2$ as coordinates on $\PP^{|V_2|-1}$. Given such a partition of $V$, let $\PP'(\K)$ denote the vanishing locus of the ideal $I_\K$ in $\PP^{|V_1|-1}\times \PP^{|V_2|-1}$. We can also think of $\PP'(\K)$ as a quotient: let $X$ be the vanishing locus of $\I_K$ in $\Aff^{|V|}$, and $Z$ the vanishing locus of $\{x_ix_j\}_{i\in V_1,j\in V_2}$. Then 
\[
\PP'(\K)= (X\setminus Z)//(\KK^*)^2
\] 
where $(\KK^*)^2$ acts on $x_i,i\in V_1$ by weight $(1,0)$ and on  $x_i,i\in V_2$ by weight $(0,1)$.

For natural numbers $d,e$, we denote by $\Delta_{d,e}$ the simplex $\Delta_{d+e-1}$ whose vertices have been partitioned into sets of respectively $d$ and $e$ vertices. The vertices of the simplicial complex $\K_n$ are also naturally partitioned according to whether the corresponding variable is of the form $x_{ij}$ or $y_i$. Given simplicial complexes $\K,\K'$, each of whose vertices are partitioned into two sets, the vertices of $\K*\K'$ come with a natural partition as well.

\begin{lemma}\label{lemma:unobstructed}
Let $n\geq 5$.
The Stanley-Reisner schemes
\begin{align*}
&\PP(K_n*\partial \Delta_{d_1-1}*\ldots * \partial\Delta _{d_k-1}*\Delta_{2n-4-d})\\
&\PP'(\K_n*\partial \Delta_{d_1,e_1}*\ldots * \partial\Delta _{d_k,e_k}*\Delta_{n-1-d,n-2-e})
\end{align*}
have unobstructed deformations as subschemes of projective space or products of projective space, respectively. Furthermore, for either of these schemes, the first and second cohomology groups of the structure sheaf vanish.
\end{lemma} 
\begin{proof}
By Theorem \ref{thm:unobstructed}, $\K_n$ is unobstructed, and $\partial\Delta_m$ is unobstructed since its ideal is principal. It follows from \cite[Proposition 2.3]{ilten:14a} that the join of unobstructed simplicial complexes is again unobstructed.
The claim regarding unobstructedness now follows from \cite[Theorem 3.3]{kleppe:14a} together with the fact that the Stanley-Reisner ring of the join of $\K_n$ with (boundaries of) simplices is shellable, hence Cohen-Macaulay \cite[Theorem 13.45]{miller:05a}.

The claim regarding the cohomology vanishing of the structure sheaf uses the standard isomorphism relating sheaf cohomology to local cohomology: if $Y$ is a subscheme of projective space or a product of projective spaces, then for $i\geq 1$
\[
H^i(Y,\CO_Y)\cong H_I^{i+1}(S)
\]
where $S$ is the (multi)homogeneous coordinate ring of $Y$ and $I$ the irrelevant ideal, see e.g~\cite[\S 9.5]{cox:11a}. As noted above, in our instance $S$ is Cohen-Macaulay and of dimension larger than three, so the claim follows.
\end{proof}

\section{A Gr\"obner Basis Computation}
We construct a term order $\prec$ on $S_n$ by successively refining partial orders as follows: 
\begin{enumerate}[(i)]
\item Order terms by total degree;
\item Terms containing some $y_i$ for $2\leq i \leq n-1$ are smaller than those without;
\item Terms containing $x_{(n-1)n}$ are smaller than those without;\label{item:2}
\item Order terms lexicographically by $x_i\prec y_{n-1} \prec y_{n-2} \prec \ldots \prec y_2 \prec y_1 \prec y_n$;
\item Order terms with any circular order  compatible with (\ref{item:2}) above.
\end{enumerate}

\begin{thm}\label{thm:1}
Under the term order $\prec$, the initial ideal of $J_n$ is the Stanley-Reisner ideal of $\K_n * \Delta_{2n-4}$.
\end{thm}

Before proving the theorem, we state two immediate corollaries:

\begin{cor}
Under the standard grading on $S_n$, the degree of $J_n$ is 
\[
\frac{2}{n-1} {2(n-2)\choose n-2}+\frac{1}{n-2} {2(n-3)\choose n-3}.
\]

\end{cor}
\begin{proof}
The degree of $J_n$ is the same is the degree of the initial ideal of $J_n$, which by Theorem \ref{thm:1} is equal to the number of top-dimensional faces of $\K_n$. Now, the number of top-dimensional faces of $\A_n$ is just the number of triangulations of the $n$-gon, which is equal to  $\frac{1}{n-1} {2(n-2)\choose n-2}$, see e.g.~\cite{lee:89a}. Taking the join of $\A_n$ with $S^0$ doubles the number. Stellar subdivision in the edge $f$ (with vertices $\delta_{1(n-1)}$ and $w_1$) increases the number of top-dimensional faces by the number of top-dimensional faces of $\A_n*S^0$ containing the edge $f$. But this is just the number of top-dimensional faces of $\A_n$ containing $\delta_{1(n-1)}$, that is, the number of triangulations of the $n$-gon containing the diagonal $\delta_{1(n-1)}$. The set of such triangulations is in bijection to triangulations of the $(n-1)$-gon, of which there are 
$\frac{1}{n-2} {2(n-3)\choose n-3}$
.
\end{proof}

\begin{cor}\label{cor:cm}
The ring $S_n/J_n$ is Cohen-Macaulay.
\end{cor}
\begin{proof}
Since $\K_n$ is a topological sphere and hence shellable, the claim follows from Theorem \ref{thm:1} together with \cite[Corollary 8.31 and Theorem 13.45]{miller:05a}.
\end{proof}

Before we begin the proof of Theorem \ref{thm:1}, we recall several standard notions. Fix a term order $\prec$. For polynomials $f$ and $g$ we define their \emph{S-polynomial} as
\begin{align*}  
S(f,g) =  \frac{lcm\big( LT_\prec(f), LT_\prec(g) \big)}{LT_\prec(f)}f -  \frac{lcm\big( LT_\prec(f), LT_\prec(g) \big)}{LT_\prec(f)}g .
\end{align*}
Here, $LT_\prec(h)$ denotes the leading term of a polynomial $h$. Observe that the leading terms of $f$ and $g$ cancel in $S(f,g)$ by construction. Another useful notion is a generalization of multivariate polynomial division by a set. Let $G = \{ g_i \}$ be a set of polynomials. We say that a polynomial $f$ \emph{reduces to zero modulo} $G$ if we can write
\begin{align*}
f = \sum_{g_i \in G} a_i g_i
\end{align*}
for polynomials $a_i$ that satisfy $ LT_\prec(a_ig_i)  \preceq LT_\prec(f)$. For more details see \S2.9 in \cite{cox:92a}.

\begin{proof}[Proof of Theorem \ref{thm:1}]
Let $G$ denote the set of generators of $J_n$ as described in \S \ref{sec:Q}. Under the term order $\prec$, the initial terms of $G$ are exactly the generators of the Stanley-Reisner ideal of $\K_n * \Delta_{2n-4}$. Hence, showing that the initial ideal of $J_n$ is equal to the Stanley-Reisner ideal of $\K_n * \Delta_{2n-4}$ is equivalent to proving that $G$ is a Gr\"obner basis for $J_n$ with respect to $\prec$. 

By \cite[Theorem 2.9.3]{cox:92a}, proving that $G$ is a Gr\"obner basis is equivalent
 to showing that the S-polynomials of pairs of elements of $G$ reduce to zero modulo $G$. Since $\prec$ gives by construction a circular order on the monomials involving only the $x$-variables, and the Pfaffians $\Phi_{ijkl}$ form a Gr\"obner basis with respect to any circular term order, the S-polynomials  $S(\Phi_{ijkl},\Phi_{i'j'k'l'})$ all reduce to zero modulo $G$.

Hence, it remains to show that the S-polynomials $S(f_i, f_j)$ and $S(f_r, \Phi_{ijkl})$ reduce to zero modulo $G$. Furthermore, by \cite[Proposition 2.9.4]{cox:92a} we only need to consider S-polynomials of pairs whose leading terms are not coprime. This leaves us four cases to consider, which we deal with one-by-one below.

To simplify notation, for $1\leq i,j\leq n$, and $i\neq j$ we set $f_i(j)=f_i-x_{ij}y_j$, where $x_{ij}=-x_{ji}$ if $j<i$. 

\begin{case} $S(f_i, f_j)$ with  $1 \leq i<j \leq n-2$.
We have 
\begin{align*}
S(f_i,f_j) &= x_{j n}f_i(n) - x_{in}f_j(n).
\end{align*}
By a straightforward calculation, this is equal to 
\begin{align*}
 -\sum_{r=1}^{i-1} y_r \Phi_{  r i j n }  + \sum_{r=i+1}^{j-1} y_r \Phi_{ i r j n  } - \sum_{r=j+1}^{n-1} y_r \Phi_{  i j r n }   - x_{i j}f_n,
\end{align*}
whose leading term is the same as that of $-y_1\Phi_{1ijn}$. All of the other summands have smaller leading terms. Hence, this reduces to zero modulo $G$.
\end{case}

\begin{case}
 $S(f_1, f_n)$.
We have 
\begin{align*}
S(f_1, f_n)  = y_1 f_1(n) + y_n f_n(1).
\end{align*}
A straightforward calculation shows that
\begin{align*}
 S(f_1, f_n) = -\sum_{r=2}^{n-1}y_r f_r.
\end{align*}
The leading term of this is the same as that of $-y_2f_2$. All other summands have smaller leading terms, so again this reduces to zero modulo $G$.
\end{case}

\begin{case}
$S(f_j, \Phi_{i j k n})$ with  $2 \leq j\leq n-2$ and $i < j < k \leq n-1$.

We have
\begin{align*}
S(f_{j}, \Phi_{i j k n} ) = x_{i k}f_j(n) - y_n (x_{i j}x_{k n} + x_{i n}x_{j k}).
\end{align*}
A straightforward calculation shows that
\begin{align*}
S(f_{j}, \Phi_{i j k n} ) =& -\sum_{r=1}^{i-1} y_r \Phi_{ r i j k } + \sum_{r=i+1}^{j-1} y_r \Phi_{ i r j k  } \\
&\qquad- \sum_{r=j+1}^{k-1} y_r \Phi_{ i j r k }  + \sum_{r=k+1}^{n-1} y_r \Phi_{  i j k r }  - x_{i j}f_k - x_{j k}f_i . \end{align*}

The leading term of this expression is the leading term of either $-x_{ij}f_k$ or $-x_{jk}f_i$. As in the previous two cases all other summands have smaller leading terms, and we see that this S-polynomial must reduce to zero modulo $G$.
\end{case}
\begin{case}
$S(f_{n-1}, \Phi_{ 1  j  (n-1)  n })$ with $2 \leq j \leq n-2$.

We have
\begin{align*}
S(f_{n-1},  \Phi_{ 1  j (n-1)  n } )= y_1(x_{1 j}x_{(n-1) n} + x_{1 n}x_{j (n-1)}) -x_{j n} f_{n-1}(1).
\end{align*}
Again by a straightforward calculation,
\begin{align*}
S(f_{n-1},  \Phi_{ 1  j  (n-1)  n } ) =& -\sum_{r=2}^{j-1}y_r \Phi_{  r j (n-1) n }  \\
&\qquad+ \sum_{r=j+1}^{n-2}y_r\Phi_{  j r (n-1) n  }  -x_{(n-1) n} f_j - x_{j (n-1)}f_n.
\end{align*}
The leading term of this agrees with the leading term of $ -x_{j (n-1)} f_n $, and all other summands have smaller leading terms. Hence, this S-polynomial also reduces to zero modulo $G$.
\end{case}
\end{proof}

\section{Syzygies}\label{sec:syz}

In the following section, we will need to understand the syzygies of the ideal $J_n\subset S_n$. However, our Gr\"obner basis computation from the proof of Theorem \ref{thm:1} essentially hands us these syzygies:

\begin{thm}
The syzygies of the ideal $J_n$ are generated by the syzygies of $I_{2,n}$, the $n\choose 3$ syzygies
\begin{align*}
R_{i j k}:\qquad0&= x_{i j}f_k - x_{i k}f_j + x_{j k}f_i \\
&\qquad+ \sum_{r=1}^{i-1} y_r \Phi_{ r i j k } - \sum_{r=i+1}^{j-1} y_r \Phi_{ i r j k } + \sum_{r=j+1}^{k-1} y_r \Phi_{ i j r k } - \sum_{r=k+1}^{n} y_r \Phi_{  i j k r } 
\end{align*}
and the syzygy
\[0=\sum_{i=1}^n y_if_i.\]
\end{thm}
\begin{proof}
Observe that these syzygies arise in the proof of Theorem \ref{thm:1} when we showed that the S-polynomials reduced to zero modulo $G$. On the other hand, these syzygies (along with the syzygies of $I_{2,n}$) are enough to generate all syzygies of $J_n$. Indeed, given an ideal $I$ with Gr\"obner basis $G$, by  \cite[Proposition 2.9.8]{cox:92a} the syzygies of $I$ are generated by the expressions showing that S-polynomials of pairs of elements of $G$ reduce to zero modulo $G$. Furthermore, it suffices to consider only those pairs of elements with relatively prime leading terms, see e.g.~\cite[Theorem 15.10]{eisenbud:95a}.
\end{proof}

\begin{rem}
We will also need to understand some of the syzygies among the generators of $I_{2,n}$.
Choose indices $1\leq i<j<k<l<m\leq n$, and let $M$ be the $5\times 5$ skew-symmetric matrix with entries $x_{ij}, x_{ik}, \ldots$ etc. Set  
\begin{align*}
\mathbf{v} = (\Phi_{j k l m}, -\Phi_{i k l m}, \Phi_{i j l m}, -\Phi_{i j k m}, \Phi_{i j k m})^T .
\end{align*}
It is easily verified that $M\mathbf{v} = \mathbf{0}$, so that the entries of $M\mathbf{v}$ are syzygies among the generators of $I_{2,n}$. More explicitly, this gives the $5\cdot {n \choose 5}$ syzygies
\begin{align*}
R_{ijklm}^r\colon\qquad 0=x_{ri}\Phi_{j k l m} -x_{rj}\Phi_{i k l m}+x_{rk} \Phi_{i j l m}, -x_{rl}\Phi_{i j k m}+x_{rm}\Phi_{i j k m}
\end{align*}  
Here, $r\in \{i,j,k,l,m\}$ and we have the convention that $x_{ab}=-x_{ba}$ and $x_{aa}=0$.
For $n>5$, these syzygies do not generate the syzygy module of $I_{2,n}$, but in fact as we shall see in the following section, they suffice to show that the coordinate ring of $\GR$ is rigid.
\end{rem}

\section{Rigidity Results}
Let $S$ be any smooth finitely generated $\KK$-algebra, and $I\subset S$ any ideal. The \emph{first cotangent cohomology} $T^1(S/I)$ of $S/I$ can be defined as the cokernel of the map
\[
\Der_\KK(S,S) \to \Hom_S(I,S/I),
\]
where $\Der_\KK(S,S)$ is the $S$-module of $\KK$-linear derivations from $S$ to itself, and an element $\partial\in\Der_\KK(S,S)$ is mapped to the homomorphism sending $f\in I$ to $\partial(f)+I$.
The vector space $T^1(S/I)$ controls the first-order deformations of $S/I$. In particular, if it vanishes, we say $S/I$ is \emph{rigid}, in which case it has no infinitesimal deformations. See e.g. \cite{hartshorne:10a} for an introduction to deformation theory.

Our second main result is the following:
\begin{thm}\label{thm:2}
For $n\geq 6$, $S_n/J_n$ is rigid.
\end{thm}

Our strategy for proving this rigidity result will be to first reduce any $\rho\in \Hom_{S_n}(J_n,S_n/J_n)$ by certain derivations. We will then use well-chosen syzygies of $J_n$ to show that $\rho$ must vanish on all generators of $J_n$. This argument will rely on our description of the standard monomials in $S_n/J_n$ coming from Theorem \ref{thm:1}. Recall that a monomial of a polynomial ring $S$ is a \emph{standard monomial} for an ideal $I$ if it is not in the initial ideal of $I$; the images of the standard monomials form a vector space basis for $S/I$. A polynomial $f\in S$ is in \emph{normal form} if it is a linear combination of standard monomials. The normal form of an element $g\in S/I$ is the unique representative $f \in S$ of $g$ which is in normal form.

\begin{lemma}\label{lemma:normal}
Fix $1\leq \alpha \leq \beta \leq n$, and set
\begin{align*}
C&=\{m\in\NN\ |\ \alpha\leq m \leq \beta\}\\
D&=\{m\in \NN\ |\  m\leq \alpha\ \mathrm{or}\ m\geq \beta\}.
\end{align*}
Now, consider any monomial $z\in \KK[x_{ij}\ |\ 1\leq i<j\leq n]$ which can be written as a product of monomials $z=z_C\cdot z_D$ satisfying:
\begin{enumerate}
\item If $x_{ij}|z_C$, then $i,j\in C$; 
\item If $x_{ij}|z_D$, then $i,j\in D$; 
\item $z_D$ is in normal form with respect to $I_{2,n}$ and the circular term order. 
\end{enumerate}
Then each monomial appearing in the normal form of $z$ is of the form $z_C'\cdot z_D$, where $z_C'$ is a product of some variables $x_{ij}$ with $i,j\in C$.
\end{lemma}
\begin{proof}
Consider any monomial $z$ as above. If it is already in normal form, the claim is obvious. Otherwise, there must be $\alpha\leq i<j<k<l\leq \beta$ with $x_{ik}x_{jl}$ dividing $z_C$. Reducing $z_C$ by the Pfaffian $\Phi_{ijkl}$ leads to a binomial $u-v$, where $u\cdot z_D$ and $v\cdot z_D$ both satisfying the assumptions of the lemma. The claim now follows by induction with respect to the circular term order.
\end{proof}

\begin{proof}[Proof of Theorem \ref{thm:2}]
Fix some $\rho\in \Hom_{S_n}(J_n,S_n/J_n)$. For ease of notation, we set $\Psi_{ijkl}$ to be the normal form of $\rho(\Phi_{ijkl})$ in $S_n$. Likewise, we set $g_i$ to be the normal form of $\rho(f_i)$.

Now, by adding derivations of the form $f\cdot \frac{\partial}{\partial x_{ij}}$ for $1\leq i<j\leq 4$ to $\rho$, we can assume that no term of $\Psi_{1234}$ is divisible by $x_{ij}$ for $1\leq i<j\leq 4$. Likewise, by adding derivations of the form $f\cdot \frac{\partial}{\partial x_{ij}}$ for $i=1,2,3$ and $4<j\leq n$, we can assume that no term of $\Psi_{123j}$ is divisible by $x_{12},x_{13},x_{23}$. Finally, by adding derivations of the form $f\cdot \frac{\partial}{\partial x_{ij}}$ for $3<i<j\leq n$, we can assume that no term of $\Psi_{12ij}$ is divisible by $x_{12}$.

We will now show that for such a homomorphism $\rho$, we must have 
\begin{equation}\label{eqn:vanishing}
\Psi_{ijkl}=0\qquad \textrm{for all}\qquad 1\leq i < j < k < l\leq n.
\end{equation}
We begin by claiming that $\Psi_{2345}=0$. Considering the syzygy $R_{12345}^2$ we must have
\[
x_{12}\Psi_{2345}-x_{23}\Psi_{1245}+x_{24}\Psi_{1235}-x_{25}\Psi_{1234}\in J_n.
\]
Note that the first two terms in this sum are already in normal form with respect to $J_n$. Set $\alpha=2$ and let $\beta$ be the largest $j\leq n$ such that $x_{ij}$ divides some term of $\Psi_{1235}$ for some $i\geq 3$. By our assumptions on $\Psi_{1235}$, we can apply Lemma \ref{lemma:normal}
to conclude that each monomial in the normal form of $x_{24}\Psi_{1235}$ with respect to $I_{2,n}$ is not divisible by $x_{12}$. Furthermore, this normal form is equal to the normal form with respect to $J_n$, since the only way a term divisible by $x_{in}y_n$ $i=2,\ldots,n-2$,  $x_{1(n-1)}y_1$, or $x_{1(n-1)}y_1$ can appear is if some term in $x_{24}\Psi_{1235}$ is already divisible by such a monomial, which is impossible. Set again $\alpha=2$ and now let $\beta$ be the largest $j\leq n$ such that $x_{ij}$ divides some term of $\Psi_{1234}$ for some $i\geq 3$. An application of Lemma \ref{lemma:normal}
shows that each monomial in the normal form of $x_{25}\Psi_{1234}$ with respect to $I_{2,n}$ is not divisible by $x_{12}$. As before, no term in this normal form is divisible by $x_{in}y_n$ $i=2,\ldots,n-2$,  $x_{1(n-1)}y_1$, or $x_{1(n-1)}y_1$ unless some term in $x_{25}\Psi_{1234}$ is already divisible by such a quadric. But this is impossible unless $n=5$, and we have assumed otherwise. Rewriting each term in the above syzygy in its normal form with respect to $J_n$, we see that no monomial coming from the latter three terms is divisible by $x_{12}$. Hence, we must have $\Psi_{2345}=0$.

Next, we show that $\Psi_{1345}=0$. Indeed, the syzygy $R_{12345}^3$ implies 
\[
x_{13}\cdot 0-x_{23}\Psi_{1345}+x_{34}\Psi_{1235}-x_{35}\Psi_{1234}\in J_n.
\]
Judicious use of Lemma \ref{lemma:normal} as above implies that the normal forms with respect to $J_n$ of $x_{34}\Psi_{1235}$ and $x_{35}\Psi_{1234}$ contain no terms divisible by $x_{23}$, hence, $\Psi_{1345}=0$. Likewise, the syzygy $R_{12345}^4$ implies 
\[
x_{14}\cdot 0-x_{24}\cdot 0+x_{34}\Psi_{1245}-x_{45}\Psi_{1234}\in J_n.
\]
Lemma  \ref{lemma:normal} as above implies that the normal form with respect to $J_n$ of $x_{45}\Psi_{1234}$ contains no term divisible by $x_{34}$, hence $\Psi_{1245}=0$. But then we must have $\Psi_{1235}=\Psi_{1234}=0$ as well.

Our next step is to show that $\Psi_{234l}=0$ for all $l\geq 5$. Indeed, the syzygy $R_{1234l}^2$ implies 
\[
x_{12}\Psi_{234l}-x_{23}\Psi_{124l}+x_{24}\Psi_{123l}-x_{2l}\cdot 0\in J_n.
\]
Again considering normal forms with respect to $J_n$ and applying Lemma \ref{lemma:normal}, no terms in the normal forms are divisible by $x_{12}$ except for $x_{12}\Psi_{234l}$, implying that $\Psi_{234l}=0$. To conclude that $\Psi_{134l}=0$ as well, we use syzygy $R_{1234l}^3$ and the relation
\[
x_{13}\cdot 0-x_{23}\Psi_{134l}+x_{34}\Psi_{123l}-x_{3l}\cdot 0\in J_n
\]
in a similar fashion.

We now make the important observation that for fixed $a<b<c<d<e$, if for three of the five choices of $i,j,k,l\in\{a,b,c,d,e\}$ we have $\Psi_{ijkl}=0$, it must in fact vanish for the other two choices of $i,j,k,l$ as well. Indeed, an appropriate syzygy of the type $R_{abcde}^r$ will involve three of the vanishing terms, along with the product of $\Psi_{ijkl}$  with a variable, implying that $\Psi_{ijkl}=0$. We can immediately apply this to conclude that for $i,j,k,l\in \{1,2,3,4,e\}$, \eqref{eqn:vanishing} holds.

For $4<d<e\leq n$, using the syzygy $R_{123de}^2$ gives us that
\[
x_{12}\cdot \Psi_{23de}-x_{23}\cdot \Psi_{12de}+x_{2d}\cdot 0 -x_{2e}\cdot 0\in J_n.
\]
Since no term of $\Psi_{12de}$ is divisible by $x_{12}$, Lemma \ref{lemma:normal} implies that $\Psi_{23de}=0$. This in turn implies that $\Psi_{12de}=0$, and hence, for $i,j,k,l\in \{1,2,3,d,e\}$, \eqref{eqn:vanishing} holds.

Next we claim that for $i,j,k,l\in\{1,2,c,d,e\}$, \eqref{eqn:vanishing} holds. Indeed, this follows directly from our observation above, since $\Psi_{12cd}=\Psi_{12ce}=\Psi_{12de}=0$. But then by a similar argument, we may conclude that \eqref{eqn:vanishing} holds for all $i,j,k,l$.

It remains to be seen that we can modify $\rho$ by derivations to also obtain that $g_i=0$ for $i=1,\ldots,n$. First, we may add derivations of the form  
$f\cdot \frac{\partial}{\partial y_{i}}$ for $i\neq 2$ to $\rho$ to achieve that no term of $g_2$ is divisible by $x_{12}$ or $x_{2i}$, $3\leq i \leq n$. 
Secondly, we may add derivations of the form $f\cdot \frac{\partial}{\partial y_{2}}$ to achieve that no term of $g_3$ is divisible by $x_{23}$. Now, the syzygy $R_{123}$ implies that
\[
x_{23}\cdot g_1-x_{13}\cdot g_2+x_{12}g_3\in J_n.
\]
But by our assumptions on the $g_i$, all three of these terms must already be in normal form with respect to $J_n$, and since the latter two contain no monomials divisible by $x_{23}$, we get $g_1=0$. Similarly, $g_3=0$, and hence, $g_2=0$. Now, considering the syzygy $R_{12m}$ and substituting in for $\Psi_{ijkl}=0$ and $g_1=g_2=0$, we get that $x_{12}\cdot g_m\in J_n$ for all $m>2$. Hence $g_m=0$, completing the proof.
\end{proof}

\begin{rem}
The above proof fails if $n=5$. In that case a straightforward calculation shows that $T^1(S_n/J_n)$ is generated by the homomorphism sending 
\begin{align*}
\Phi_{2345}&\mapsto y_1\\
\Phi_{1345}&\mapsto -y_2\\
\Phi_{1245}&\mapsto y_3\\
\Phi_{1235}&\mapsto -y_4\\
\Phi_{1234}&\mapsto y_5.\\
\end{align*}
\end{rem}

\begin{rem}
The proof of Theorem \ref{thm:2} also gives an algebraic proof that the coordinate ring of $\GR$ in its Pl\"ucker embedding is rigid for $n\geq 5$. Svanes had previously shown \cite{svanes:75a} that (with the exception of $G(2,4)$) the coordinate ring of any Grassmannian in its Pl\"ucker embedding is rigid by using vanishing results of certain cohomology groups on the Grassmannian.
\end{rem}

\section{Geometric Results}\label{sec:geom}
There are two natural projective varieties which we can associate to the ideal $J_n$: the vanishing locus $X$ of $J_n$ in $\PP^{ {n \choose 2}+n-1}$, and the vanishing locus $X'$ of $J_n$ in $\PP^{n\choose 2}\times \PP^{n-1}$. In the case of $X$, we are viewing $S_n$ and $J_n$ as having the standard grading, whereas in the case of $X'$ we are using the bigrading $\deg x_{ij}=(1,0)$, $\deg y_i=(0,1)$.
Concretely, we have
\begin{align*}
&X=(\spec S_n/J_n\setminus Z)//\KK^*
\qquad \qquad&&X'=(\spec S_n/J_n\setminus Z')//(\KK^*)^2
\end{align*}
where $Z$ is the origin in $\spec S_n=\Aff^{n\choose 2}\times \Aff^n$, and $Z'$ is the union of $\{0\}\times \Aff^n$ with $\Aff^{ n\choose 2} \times \{0\}$.
Note that the variety $X'$ is nothing other than the projectivization of the bundle $\Q^*$. In particular, $X'$ is non-singular.
On the other hand, The variety $X$ comes with a natural rational map $\phi:X\dashrightarrow \GR$. The locus where $\phi$ is defined is simply the bundle $\Q^*$; the indeterminacy locus of $\phi$ is a copy of $\PP^{n-1}$. A straightforward calculation shows that this indeterminacy locus is exactly the singular locus of $X'$.

Our first geometric result is that relative complete intersections in $X$ and $X'$ degenerate to particularly nice Stanley-Reisner schemes: 
\begin{prop}\label{prop:g1}
Fix $0\leq k \leq 2n-3$, and consider generic hypersurfaces \[H_1,\ldots,H_k\subset \PP^{ {n\choose 2}+n-1}\] satisfying $d=\sum_{i=1}^k d_i\leq 2n-3$, where $d_i=\deg H_i$ . Then the projective variety 
\[
X\cap H_1\cap \ldots \cap H_k
\]
has anticanonical bundle $\CO(2n-3-d)$ and hence is Fano if and only if $d< 2n-3$. Furthermore, this variety degenerates to 
\[\PP(\K_n*\partial \Delta_{d_1-1}*\ldots * \partial\Delta _{d_k-1}*\Delta_{2n-4-d})\]
 and has unobstructed deformations as a subvariety of projective space. Finally, this variety is non-singular if and only if $k>n-1$.
\end{prop}

\begin{prop}\label{prop:g2}
For $0\leq k \leq 2n-3$, consider generic hypersurfaces \[H_1,\ldots,H_k\subset \PP^{n\choose 2}\times \PP^{n-1}\] satisfying $d=\sum_{i=1}^k d_i\leq n-1$
and $e=\sum_{i=1}^k e_i\leq n-2$, where $(d_i,e_i)$ is the bidegree of $H_i$.
Then the projective variety 
\[
X\cap H_1\cap \ldots \cap H_k
\]
is non-singular with anticanonical bundle $\CO(n-1-d,n-2-e)$, and hence is Fano if and only if $d<n-1$ and $e<n-2$. Furthermore, this variety degenerates to
\[\PP'(\K_n*\partial \Delta_{d_1,e_1}*\ldots * \partial\Delta _{d_k,e_k}*\Delta_{n-1-d,n-2-e})\]
  and has unobstructed deformations as a subvariety of $\PP^{n\choose 2}\times \PP^{n-1}$.
\end{prop}

\begin{proof}[Proof of Propositions \ref{prop:g1} and \ref{prop:g2}]
The bi-graded coordinate ring of $X'$ is just $S_n/J_n$ with grading on the variables as discussed above. By Theorem \ref{thm:1}, this ring has a bihomogeneous degeneration to the Stanley-Reisner ring of $\K_n*\Delta_{n-1,n-2}$. Hence, we get a degeneration of $X'$ to $\PP'(\K_n*\Delta_{n-1,n-2})$. To get a degeneration of $X'\cap H_1\cap \ldots \cap H_k$, we simply degenerate the $k$ equations for $H_i$ to square-free monomials of the same bidegree with disjoint support, and also disjoint from the support of $I_{\K_n}$. The resulting specialization is exactly \[\PP'(\K_n*\partial \Delta_{d_1,e_1}*\ldots * \partial\Delta _{d_k,e_k}*\Delta_{n-1-d,n-2-e}).\]
The underlying family is flat since it is a relative complete intersection in a flat family.
To get a degeneration of $X\cap H_1\cap \ldots \cap H_k$, we argue similarly, but forgetting the bidegree.

The Stanley-Reisner schemes occurring as the special fibers in these degenerations are unobstructed (see Lemma \ref{lemma:unobstructed}), so it follows that $X\cap H_1\cap \ldots \cap H_k$ and 
$X'\cap H_1\cap \ldots \cap H_k$ are unobstructed as well.

Now, the Stanley-Reisner ring of $\K_n$ is Cohen-Macaulay and Gorenstein, from which it follows that the total coordinate ring of the degeneration above is also Gorenstein. On   
$\PP(\K_n*\Delta_{2n-4})$ and $\PP'(\K_n*\Delta_{n-1,n-2})$ we respectively have canonical bundles
$\CO(3-2n)$ and $\CO(1-n,2-n)$, cf.~\cite[Proposition 2.1]{ilten:14a}. But this must also be true for the total spaces of the degeneration of $X$ and $X'$s. Indeed, the map of deformation functors from deformations of pairs to abstract deformations is injective by the exact sequence of \cite[(3.30)]{sernesi:06a} (which also holds in the singular setting with appropriate modification, see e.g.~\cite[\S 3]{altmann:10a}), and the cohomology vanishing of Lemma \ref{lemma:unobstructed}. Hence, any line bundle on the central fiber of our degeneration lifts uniquely to the total space.
The claim concerning the anticanonical bundle now follows by adjunction, cf. \cite[Proposition 2.4]{altman:70a}.

Finally, the claim in Proposition \ref{prop:g1} regarding non-singularity is a straightforward application of Bertini's theorem.
\end{proof}

\begin{rem}\label{rem:g}
Let $Y$ be a $(3n-2-k)$-dimensional toric Fano variety whose moment polytope is the reflexive polytope $P$, see e.g.~\cite{cox:11a}. Consider a generic Fano variety $V$ of the same deformation type as  
\[
X\cap H_1\cap \ldots \cap H_k
\]
from Proposition \ref{prop:g1}. This proposition implies a combinatorial criterion that guarantees the existence of a degeneration from $V$ to $Y$. Indeed, such a degeneration exists if $P$ admits a unimodular regular triangulation whose underlying simplicial complex is of the form $\K_n*\partial \Delta_{d_1-1}*\ldots * \partial\Delta _{d_k-1}*\Delta_{2n-4-d}$, see e.g. \cite[Proposition 3.3]{ilten:14a}. Proposition \ref{prop:g2} may be used in a similar fashion to find degenerations to toric varieties embedded in products of projective spaces. We leave the details to the reader.
\end{rem}
\begin{ex}
Consider the lattice polytope $Q\subset \RR^5$ whose vertices are given by the columns of the following matrix:
\begin{align*}
\left(\begin{array}{c c c c c c c c c c c c c}
1& 1& 0& -1& 0& 1& 0& 0& -1& 0 &0&1&0\\ 
0& -1& 1& 1& 1& 1& 2& -1& -1&0&0&-1&0\\ 
0& 0& 0& 0& -1& -1& -1& 1& 1&0&0&0&0\\
0& 0& 0& 0&  0&  0&  0& 0& 0&1&-1&1&0\\
2& 4& 4& 4&  4&  3&  4& 3& 4&4& 4&4&0\\
\end{array}\right)
\end{align*}
Let $\phi:\RR^5\to \RR^4$ be the projection onto the first four factors. Then $P=\phi(Q)$ is a four-dimensional reflexive polytope, and one easily checks that the regular unimodular triangulation given by the images of the lower facets of $Q$ is equal to $\K_6*\Delta_0$ as an abstract simplicial complex. Hence, by the above remark, a Fano variety with the same deformation type as a codimension $11$ linear section of $X$ degenerates to the toric variety whose moment polytope is $P$.
\end{ex}

Especially in light of Remark \ref{rem:g}, it is interesting to ask if $X\cap H_1\cap \ldots \cap H_k$ or $X'\cap H_1\cap \ldots H_k$ is generic for its deformation type. We can answer this question in the affirmative when intersecting with $0$ or $1$ hypersurface:
\begin{thm}\label{thm:g}
Assume that $n>5$.
The varieties $X$ and $X'$ are rigid. Given a hypersurface $H$ in either $\PP^{ {n\choose 2}+n-1}$ or $\PP^{n\choose 2}\times \PP^{n-1}$ not containing $X$ or $X'$, any deformation of $X\cap H$ or $X'\cap H$, respectively, is again the intersection of $X$ or $X'$ with a hypersurface $\widetilde H\in |H|$. 
\end{thm}

\begin{proof}
Let $Y$ be either $X$ or $X'$. Then $H^2(Y,\CO_Y)=0$ by the degeneration from Proposition \ref{prop:g1} or \ref{prop:g2}, semicontinuity of cohomology, and the second claim of Lemma \ref{lemma:unobstructed}.
Furthermore, since $S_n/J_n$ is Cohen-Macaulay (cf.~Corollary \ref{cor:cm}), it follows that the depth of the irrelevant ideal (the ideal of $Z$ or $Z'$) on $S_n/J_n$ is equal to the dimension of $S_n/J_n$, which is at least $12$. The claim of the Theorem now follows by combining Theorem \ref{thm:2} with Lemma \ref{lemma:comparison} below, where we apply the lemma with $V$ equal to $\PP^{ {n\choose 2}+n-1}$ or $\PP^{n\choose 2}\times \PP^{n-1}$, and $l=1$.
\end{proof}

\begin{lemma}[cf. {\cite[Proposition 5.2]{ilten:15a}}]\label{lemma:comparison}
Let $V$ be a smooth complete toric variety with homogeneous coordinate ring $C$ and irrelevant ideal $I$. Let $Y$ be a subscheme of $V$ corresponding to the ideal $J\subset C$, and set $A=C/J$. 
Consider a relative complete intersection $W=V(f_1,\ldots,f_l)\subset Y$ for homogeneous elements $f_i\in C$. Set $B=A/\langle f_1,\ldots,f_k\rangle$. If 
\begin{enumerate}
\item $\depth_I B\geq 3$;
\item $H^2(Y,\CO_Y)=0$; and
\item $T^i(A)=0$ for $1\leq i \leq l+1$
\end{enumerate} 
then any deformation of $W$ is again a relative complete intersection in $Y$.
\end{lemma}
\begin{proof}
We will show that the forgetful map $\Def_{W/Y} \to \Def_W$ from the local Hilbert functor of $W$ in $Y$ to the deformation functor of $W$ is smooth, which implies the claim. But this map is smooth if $T^1(Y,\CO_W)=0$. Now, Lemmas 4.2 and 4.3 and Proposition 4.24 from \cite{kleppe:14a} show that the first two conditions imply a surjection $T^1(A,B)^H\to T^1(Y,\CO_W)$, where $H$ is the Picard torus of $V$.
Note that we are using the embedding of $Y$ in $V$ to ensure that $Y$ is a geometric quotient.
 Now, the third claim implies the vanishing of $T^1(A,B)$, see e.g. \cite[Lemma 5.1]{ilten:15a}.
\end{proof}

\begin{rem}Results for higher codimension relative complete intersections similar to those of Theorem \ref{thm:g} would follow if we could show the vanishing of certain higher cotangent cohomology modules of $S_n/J_n$. The details are left to the reader.
\end{rem}

\bibliographystyle{alpha}
\bibliography{sr-quotientbundle}
\end{document}